\documentclass[12pt]{amsart}
\newcommand{\al}{\alpha}

\newcommand{\de}{\delta}
\newcommand{\ep}{\varepsilon}
\newcommand{\T}{{\Theta}}
\newcommand{\partialtau}[2]{\partial_{\tau_{#1 #2}}}

\newcommand\tc[2]{\theta\left[\begin{smallmatrix}#1\\ #2\end{smallmatrix}\right]}
\newcommand{\sm}[2]{\left(\begin{smallmatrix}#1\\#2\end{smallmatrix}\right)}

\newcommand{\bes}{\begin{equation*}}
\newcommand{\ees}{\end{equation*}}

\newcommand{\smaq}{\left[ \begin{smallmatrix}}
\newcommand{\smcq}{\end{smallmatrix}\right]}
\newcommand{\smat}{\left( \begin{smallmatrix}}
\newcommand{\smct}{\end{smallmatrix}\right)}

\def\pmatrix{\left(\begin{matrix}}
\def\endpmatrix{\end{matrix}\right)}

\newcommand{\CC}{{\mathbb{C}}}

\newcommand{\HH}{{\mathcal{H}}}
\newcommand{\FF}{{\mathbb{F}}}
\newcommand{\EE}{{\mathbb{E}}}
\newcommand{\PP}{{\mathbb{P}}}

\newcommand{\RR}{{\mathbb{R}}}
\newcommand{\ZZ}{{\mathbb{Z}}}

\newcommand{\NN}{{\mathbb{N}}}

\newcommand{\calO}{{\mathcal O}}
\newcommand{\calH}{{\mathcal H}}

\newcommand{\calA}{{\mathcal A}}

\newcommand{\op}{\operatorname}
\newcommand{\Sp}{\op{Sp}}

\newcommand{\GL}{\op{GL}}

\newcommand{\Sym}{\op{Sym}}

\newcommand{\odd}{\op{odd}}

\newcommand\grad{\op{grad}}
\newcommand\tr{\op{Tr}}

\theoremstyle{plain}
\newtheorem{thm}{Theorem}
\newtheorem{lm}[thm]{Lemma}
\newtheorem{prop}[thm]{Proposition}

\theoremstyle{definition}

\newtheorem{rem}[thm]{Remark}

\newcommand\tch[2]{{\theta\left[\begin{matrix}{#1}\\ {#2}\end{matrix}\right]}}
\newcommand\stch[2]{{\theta\left[\begin{smallmatrix}{#1}\\ {#2}\end{smallmatrix}\right]}}
\newcommand{\ch}[2]{\left[\begin{smallmatrix}#1\\#2\end{smallmatrix}\right]}

\begin{document}
\title{Vector-valued modular forms and the Gauss map  }
\author[Dalla Piazza]{Francesco Dalla Piazza}
\email{dallapiazza@mat.uniroma1.it}
\author[Fiorentino]{Alessio Fiorentino}
\email{fiorentinoalessio@alice.it }
\author[Grushevsky]{Samuel Grushevsky}
\email{sam@math.sunysb.edu}
\thanks{Research of the third author is supported in part by National Science Foundation under the grant DMS-12-01369.}
\thanks{Research of the remaining authors is supported in part by PRIN  and Progetto di Ateneo dell' Universit\`a  La Sapienza: "Spazi di Moduli e Teoria di Lie"}
\author[Perna]{Sara Perna}
\email{perna@mat.uniroma1.it}
\author[Salvati~Manni]{Riccardo Salvati Manni}
\email{salvati@mat.uniroma1.it}
\address{Mathematics Department, Stony Brook University,
Stony Brook, NY 11794-3651, USA}
\address{Dipartimento di Matematica, Universit\`a di Roma ``La Sapienza'', Piazzale Aldo Moro, 2, I-00185 Roma, Italy}

\begin{abstract}
We use the gradients of theta functions at odd two-torsion points --- thought of as vector-valued modular forms --- to construct holomorphic differential forms on the moduli space of principally polarized abelian varieties, and to characterize the locus of decomposable abelian varieties in terms of the Gauss images of two-torsion points.
\end{abstract}
\maketitle

\section*{Introduction}
The geometry of Siegel modular varieties --- the quotients of the Siegel upper half-space $\calH_g$ by discrete groups --- has been under intense investigation for the last forty years, with various results about their birational geometry, compactifications, and other properties. Some of the first results in this direction are due to Freitag, who in \cite{freitagholdiff1,freitagholdiff2} showed that some Siegel modular varieties are not unirational   by constructing non-zero differential forms on them. This proof requires two ingredients: suitably compactifying the variety and arguing that the differential form extends, and actually constructing the differential forms. Freitag proved the appropriate general extension result for differential forms.  Thanks to \cite{amrtbook} and \cite{Tai} and much subsequent work the theory of compactifications of locally symmetric domains and the extension of differential forms is now well-known in full generality.

In this paper we focus on the original problem of constructing differential forms on Siegel modular varieties.  We recall that differential forms on Siegel modular varieties can be constructed from suitable vector-valued modular forms. In general vector-valued modular forms can be defined by theta series with pluriharmonic coefficients, but in general the question of whether the series thus constructed are identically zero is very complicated.  General results on the existence and  non-vanishing of holomorphic differential forms can be found in \cite{weissauervector} and \cite{weissauerdivisors}.  In  connection  with the possibility of finding special divisors in the Siegel modular varieties in the sense of Weissauer \cite{weissauerdivisors}  we will restrict our attention to non-zero differential forms of degree one less than the top.

Freitag in \cite{freitagkorper} constructed such forms on $\calA_g$ for $g\equiv 1\pmod 8$, for $g\ge 17$, while the fifth author in \cite{smholdiff} gave a completely different construction for $g\equiv 1\pmod 4$, $g\ne1,5,13$. In this paper we present an easier and more natural method of constructing such differentials forms, providing also a natural bridge between methods of \cite{freitagkorper} and \cite{smholdiff}. Our tools will be the gradients of theta functions and expressions in terms of them considered by the third and fifth author in \cite{grsmodd1,grsmodd2}. Our result is the following.

Denote by $\partial:=\left( \frac{(1+\delta_{ij})}{2}\partial_{\tau_{ij}}\right)$ the matrix of partial derivatives with respect to $\tau$. Let $f,h$ be two scalar modular forms of the same weight for some modular group $\Gamma$ acting on $\calH_g$. Then $A:=h^2\partial(f/h)$ is a matrix-valued modular form. Denote by $A^{ad}$ the adjoint matrix of $A$ (the transpose of the matrix of cofactors), and denote by $d\check{\tau}_{ij}$ the wedge product of all $d\tau_{ab}$ for $1\le a\le b\le g$ except $d\tau_{ij}$, with the suitable sign. Denote by $d\check\tau$ the matrix of all $d\check{\tau}_{ij}$. Then
\begin{thm}\label{thm:holdiff} Let $g \geq 2$, let $f:=\T[\ep](\tau)$ and $h:=\T[\de](\tau)$ be second order theta constants. Then the modular form constructed as above,
$$
 \omega:=\tr(A_{\ep, \de}^{ad} d\check\tau)
$$
is a non-zero holomorphic differential form on $\calA_g(\Gamma):=\calH_g/\Gamma$ in degree one less than the top (i.e.~of degree $g(g+1)/2-1$). Here for $g$ odd we have $\Gamma=\Gamma_g(2,4)$, while for $g$ even it is an index two subgroup $\Gamma_g^*(2,4)\subset\Gamma_g(2,4)$.
\end{thm}

In what follows we will discuss the relation of special cases of this construction to those of Freitag \cite{freitagholdiff2} and the fifth author \cite{smholdiff}.  In a related direction,  we revisit the method of constructing vector-valued modular forms using gradients of odd theta functions with half integral characteristics.  Recall that the gradients at $z=0$ of  odd theta functions with  half integral characteristic can be thought of as the images of two-torsion points that are smooth points of the  theta divisor under the Gauss map. In this direction, we obtain an analytic proof of the following geometric statement.
\begin{thm}\label{thm:decomposable}
A principally polarized abelian variety is decomposable (i.e.~is a product of lower-dimensional ones) if and only if the  images under the Gauss map of  all smooth two-torsion points in the theta divisor lie on a quadric in $\PP^{g-1}$.
\end{thm}

The structure of the paper is as follows.  In section 1 we recall some basic  facts about theta functions and vector-valued  modular forms. In section 2 we collect  several results about gradients of odd theta  functions. In section 3 we  prove Theorem
\ref{thm:decomposable}.  In section 4 we recall and improve results of Freitag and the fifth author about  holomorphic  differential forms on the Siegel  varieties. Finally, in section 5 we prove theorem \ref{thm:holdiff} and explain the relation among the approaches to constructing differential forms on Siegel modular varieties.

\section*{Acknowledgements}
The third author would like to thank Universit\`a Roma La Sapienza for hospitality in March 2015, when some of the work for this paper was completed.

\section{Definitions and notation}
We use the standard definitions and notation in working with complex principally polarized abelian varieties (ppav), as used in \cite{grsmodd1}, which we now quickly summarize.

\subsection{Siegel modular forms}
Let $\HH_g$ be the Siegel upper-half-space of degree $g$, namely the space of $g \times g$ complex symmetric matrices with positive definite imaginary part.
The symplectic  group $\Sp(2g, \RR)$ acts transitively on $\calH_g$ as
$$
 \gamma\cdot\tau= (A\tau+B) ( C\tau+D)^{-1}\quad {\rm where}\quad \gamma=\pmatrix A&B\cr
  C&D\endpmatrix,
$$
where $A,\,B,\,C,\,D$ are the $g\times g$ blocks of the matrix $\gamma$. We will keep this block notation for a symplectic matrix throughout the paper.

The Siegel modular group is $\Gamma_g:=\Sp(2g,\ZZ)$.
The principal congruence subgroup of level $n\in\NN$ is defined as:
$$
\Gamma_g(n):=\left\lbrace \gamma\in\Gamma_g\,|\,\gamma\equiv1_{2g}\ {\rm mod}\ n\right\rbrace.
$$
A subgroup of finite index in $\Gamma_g$ is called a congruence subgroup of level $n$ if it contains $\Gamma_g(n)$. Notice that if $g>1$ every subgroup of finite index is a congruence subgroup.
The Siegel modular varieties obtained by taking the quotients with respect to the action of congruence subgroups are of central importance in the theory of principally polarized abelian varieties (ppav), as they define moduli spaces of ppav with suitable level structures.

More precisely, an element $\tau\in\calH_g$  defines the complex abelian variety $X_\tau:=\CC^g/\ZZ^g+\tau
\ZZ^g$, hence $\tau$ is usually called a period
matrix of the abelian variety $X_{\tau}$.
The quotient of $\calH_g$ by
the action of the Siegel modular group is classically known to be the moduli space of ppav: $\calA_g:=\HH_g/\Gamma_g$.

We will use the so-called theta groups, which are congruence subgroups of level $2n$ defined as
$$
 \Gamma_g(n,2n):=\left\lbrace \gamma \in\Gamma_g(n)\, |\, {\rm
 diag}(A^tB)\equiv{\rm diag} (C^tD)\equiv0\ {\rm mod}\
 2n\right\rbrace,
$$
and will also need the level 4 congruence subgroup
\begin{equation}\label{gamma24*}
\Gamma_g^*(2,4):=\left\lbrace \gamma\in\Gamma_g(2,4)\, |\,{\rm Tr}(A-1_g) \equiv 0 \, {\rm mod} \, 4 \right\rbrace,
\end{equation}
which is of index $2$ in $\Gamma_g(2,4)$. From now on, we will assume $g>1$ and
denote by $\Gamma$ an arbitrary congruence subgroup of $\Gamma_g$. We denote $N:=g(g+1)/2$, so that $\calA_g(\Gamma):=\HH_g/\Gamma$ is a complex $N$-dimensional orbifold.

Let $\rho:\GL(g,\CC)\to\operatorname{End}(V)$ be an irreducible finite-dimensional rational representation; such representations are characterized by their highest  weight
$(\lambda_1, \lambda_2, \dots, \lambda_g)\in \ZZ^{g}$, with $\lambda_1\ge\dots\ge\lambda_g$.  It will also be convenient for us to allow half-integer weights, which means  to  consider also  $\det^{1/2}\otimes \rho '$  for a representation $\rho'$ with integer weight. Let then $[\Gamma, \rho]$ be the space of holomorphic functions
$f:\calH_g\rightarrow V_{\rho}$ satisfying:
$$
 [\Gamma, \rho]:=\lbrace f:\calH_g\rightarrow V_{\rho}\,\mid\, f(\gamma\cdot\tau)=\rho(C\tau+D)f(\tau), \quad \forall \gamma\in\Gamma,\forall\tau\in\calH_g\rbrace.
$$
Such a function $f$ is called a vector-valued modular form or $\rho$-valued modular form
with respect to the representation $\rho=(\lambda_1, \lambda_2, \dots, \lambda_g)$ and the group $\Gamma$. We call $\lambda_g$ the \emph{weight} of the vector-valued modular form $f$.

Since $\calH_g$ is contractible, a $\rho$-valued modular form is a holomorphic section of a corresponding vector bundle on $\calA_g(\Gamma)$. Denoting by $\EE$ the rank $g$ vector bundle over $\calA_g$ whose fiber over $A$ is the space $H^{1,0}(A,\CC)$, sections of  $\EE$ are modular forms for the standard representation of $\GL(g,\CC)$ on $\CC^g$ and for the group $\Gamma_g$.\\

More generally it is possible to define a  vector-valued modular form with a multiplier system for this kind of representation, see \cite{freitagbooksingular}  for details. We will  make use of them if  necessary.

\subsection{Theta functions}
Many examples of modular forms can be constructed by means of the so-called theta functions. Denote by $\FF_2=\ZZ/2\ZZ$.
For $\ep,\de\in \FF_2^g$ the theta function with characteristic $m=[\ep,\de]$ is the holomorphic function $\theta_m:\calH_g\times \CC^g\to\CC $ defined by the series:
$$
\theta_m(\tau, z):=\sum\limits_{p\in\ZZ^g} e^{ \pi i \left[\left(
p+\ep/2\right)^t\tau\left(p+\ep/2\right)+2\left(p+\ep/2\right)^t\left(z+\de/2\right)\right]}.
$$
We  shall write $\tc {\ep}{\de}(\tau,z)$  for $\theta_m(\tau, z)$  if we  need to emphasize the dependence on the characteristics.
The characteristic $m$ is called even or odd depending on whether the scalar product $\ep\cdot\de\in \FF_2$ is
zero or one, and the corresponding  theta function is even or odd
in $z$, respectively. The number of even (resp. odd) theta characteristics is
$2^{g-1}(2^g+1)$ (resp. $2^{g-1}(2^g-1)$).
Furthermore, theta functions with characteristics are solutions of the heat equation:
\begin{equation}\label{heat equation}
\frac{\partial^2}{ \partial z_ i\partial z_j}\theta_m(\tau,z)=2\pi i (1+\delta_{ij})\frac{\partial}{\partial \tau_{ ij}}\theta_m(\tau,z),\;\;1\le i,j\le g.
\end{equation}

For $\sigma\in\FF_2^g$ the corresponding theta function of second order is defined as
$$
\T[\sigma](\tau,z):=\tch{\sigma}{0}(2\tau,2z).
$$
A theta constant is the evaluation at $z=0$ of a theta function.
Throughout the paper we will drop the argument $z=0$ in the notation for theta constants.
All odd theta constants with  characteristics vanish identically in $\tau$, as the
corresponding theta functions are odd functions of $z$, and thus there
are $2^{g-1}(2^g+1)$ non-trivial theta constants. All the $2^g$ second order theta functions are even in $z$, so there are $2^g$ theta constants of the second order.

As far as we are concerned, we will focus on the behaviour of the theta constants under the action of subgroups of $\Gamma_g(2)$. By~\cite{igusabook}, we have the following transformation formula:
\begin{equation}\label{transformcostanthetagamma2n4n}
\theta_{m} (\gamma \cdot \tau) = \kappa(\gamma) e^{2 \pi i \phi_m(\gamma)} \det{(C \tau + D)}^{1/2} \theta_m (\tau) \quad \quad \forall \gamma \in \Gamma_g(2),
\end{equation}
where
\begin{equation*}\label{generalphi}
\phi_m(\gamma)  =  -\frac{1}{8} (\ep^tB^t D \ep +  \de^t A^t C \de - 2 \ep^t B^t C \de) + \frac{1}{4} \op{diag}(A^tB)^t(D \ep - C \de)
\end{equation*}
and $\kappa (\gamma)$ is an 8\textsuperscript{th} root of  unity, with the same sign ambiguity as $\det{(C \tau + D)}^{\frac{1}{2}}$.

Regarding second order theta constants, we will focus on the action of subgroups of $\Gamma_g(2,4)$. For every $\gamma\in\Gamma_g(2,4)$ let $\tilde{\gamma}\in\Gamma_g$ be such that $2(\gamma\cdot\tau)=\tilde{\gamma}\cdot(2\tau)$, that is $\tilde{\gamma}=\sm{A&2B}{C/2&D}$
Hence, applying the transformation rule~\eqref{transformcostanthetagamma2n4n} to the second order theta constants we get:
\begin{equation}\label{trans second}
\Theta[\sigma](\gamma\cdot\tau)=\kappa(\tilde{\gamma})\det(C\tau+D)^{1/2}\Theta[\sigma](\tau),\;\forall\gamma\in\Gamma_g(2,4).
\end{equation}
The second order theta constants are then modular forms of weight one half with respect to the congruence subgroup $\Gamma_g(2,4)$ and  $v_\Theta(\gamma):=\kappa(\tilde{\gamma})$ is a fourth root of unity.
For a fixed $\tau\in\HH_g$, the abelian variety $X_\tau$ comes with a principal polarization given by its theta divisor $\T_\tau$, namely the zero locus of the holomorphic function $ \theta_0(\tau,z)$. One can identify, even though in a non-canonical way, the characteristic $m=[\ep,\de]\in\FF_2^g$ with the two-torsion point $x_m=(\ep\tau+\de)/2$ on the ppav $X_\tau$.  To this  divisor we associate the symmetric  line  bundle $\mathcal L=\calO_{ X_\tau}( \Theta_{\tau})$ and the  theta functions with characteristic $m$ is, up  to a constant factor, the unique section  of the line bundle $t_{x_m}^*\mathcal L$. A two-torsion point $x_m$ is called even/odd depending on whether the characteristic $m$ is even or odd.
Denoting by $X_\tau[2]$ the set of two-torsion points, note that for any $x_m\in X_\tau[2]$ we have
$\calO_{ X_\tau}( 2\Theta_{\tau})\simeq\mathcal L^{\otimes 2}\simeq(t_{x_m}^*\mathcal L)^{\otimes 2}$. Thus squares of theta functions with characteristics can be expressed in terms of a basis of sections, given by theta functions of the second order. The explicit formula is Riemann's bilinear relation:
\begin{equation}\label{bilinear}
   \stch{\ep}{\de}(\tau,z)^2=\sum\limits_{\sigma\in\FF_2^g}(-1)^{\sigma\cdot\de}\Theta[\sigma+\ep](\tau,z)\Theta[\sigma](\tau,0)
\end{equation}
Similarly, for every $\al,\,\ep\in\FF_2^g$ the following relation holds:
\begin{equation}\label{addition}
\T[\al](\tau)\T[\al+\ep](\tau)
=\frac{1}{2^g}\sum\limits_{\sigma\in(\ZZ/2\ZZ)^g}(-1)^{ \al\cdot\sigma}\tc{\ep}{\sigma} (\tau)^2.
\end{equation}
It is easily seen that the character $v_\Theta^2$ is trivial precisely on the subgroup $\Gamma_g^*(2,4)\subset\Gamma_g(2,4)$.

As we are interested in the characterization of the locus of decomposable abelian varieties we need to recall the following analytic  characterization:
\begin{thm}[\cite{sasaki},\cite{smlevel2}]\label{thm:decchar}
A ppav is indecomposable (that is, is not equal to a product of lower-dimensional ppav) if and only if the matrix
$$
M(\tau):= \left(\begin{matrix}
\dots  &\Theta[\ep]&\dots&\dots\\
\dots  &\dots &\dots&\dots\\
\dots&\partial_{\tau_{ij}}\Theta[\ep]&\dots&\dots\\
\dots&\dots&\dots&\dots\\
\end{matrix}\right)$$
(with entries taken for all $\ep\in\FF_2^g$ and for all $1\le i\le j\le g$) has maximal rank, i.e.~rank
$\frac{g(g+1)}{2}+1$.
\end{thm}
We recall also that taking the gradient with respect to $z$  of the holomorphic function $\theta_0(\tau,z)$, we get the Gauss map
$$G:\Theta_{\tau}\dashrightarrow \PP^{g-1}$$
defined on the smooth locus of the theta divisor $\Theta_\tau\subset X_\tau$. The Gauss map is
dominant if and only if the ppav $( X_{\tau},  \Theta_{\tau})$ is indecomposable (i.e.~is not a product of lower-dimensional ppav).

\smallskip
We will also have to deal with indexing by subsets of the coordinates, and fix notation for this now. For any set $X$, we denote by $P(X)$ the collection of all its subsets, and by $P_k(X)$ the collection of all its subsets of cardinality $k$. If $X\subset\ZZ$, we can view it as an order (i.e.~as a set ordered increasingly), and denote by $P_k^*(X)\subset P^*(X)$ respectively the collection of its sub-orders (i.e.~increasingly ordered subsets). If $I\in P_k^*(X)$ we denote by $I^c$ its complementary set thought of as an ordered set. Finally, we denote $X_g:=\{1,\dots,g\}$, thought of as an ordered set.

\section{Gradients of theta functions}
In \cite{grsmodd1} gradients of theta functions are used to study the geometry of the moduli space of principally polarized abelian varieties --- this study was further continued in \cite{grsmjacformula,grsmodd2,grsmconjectures,grhu1,grhu2}. Indeed, for any odd $m$ the gradient
\begin{equation}\label{graddefine}
v_m(\tau):=\grad_z\theta_m(\tau,z)|_{z=0}
\end{equation}
is a not identically zero vector-valued modular form for the group $\Gamma_g(4, 8)$   for the representation $\det^{\otimes 1/2}\otimes\operatorname{std}$, where $\mathop{std}$ is the standard representation of $\op{GL}(g,\CC)$ on $\CC^g$. We have $$v_m\in H^0(\calA_g(4,8),\det\EE^{\otimes 1/2}\otimes \EE).$$ In \cite{grsmodd1} it is shown that in fact the set of gradients of theta functions for all odd $m$ defines a generically injective map of $\calA_g(4,8)$ to the set of $g\times 2^{g-1}(2^g-1)$ complex matrices (and in fact to the corresponding Grassmannian), providing a weaker analog for ppav of a result of Caporaso and Sernesi \cite{case1,case2} characterizing a generic curve by its bitangents or theta hyperplanes.

For $\ep,\de\in\FF_2^g$ define the $g\times g$ symmetric matrix $C_{\ep\,\de}(\tau)$ with entries
\begin{equation}\label{Cdef}
  C_{\ep\,\de,ij}(\tau):=2\partial_{z_i}\stch\ep{\de}(\tau,0)\,\partial_{z_j}\stch\ep\de(\tau,0),
\end{equation}
where $\partial_{z_i}:=\frac{\partial}{\partial z_i}$. Notice that $C_{\ep\,\de}=2\,v_{\ch \ep\de}^t\,v_{\ch \ep \de}$.
Moreover, define the $g\times g$ symmetric matrix $A_{\ep\,\de}$ with entries
\begin{equation}\label{Adef}
A_{\ep\,\de,ij} (\tau):=(\partial_{z_i}\partial_{z_j}\Theta[\de](\tau))\,\Theta[\ep](\tau)-(\partial_{z_i}	\partial{z_j}\Theta[\ep](\tau))\,\Theta[\de](\tau).
\end{equation}
In the current paper it will be convenient also to write $C_{\ep\,\de}$ and $A_{\ep\,\de}$ as column  vectors of size $N=g(g+1)/2$, which we will denote ${\bf C}_{\ep\de}$ and ${\bf A}_{\ep\,\de}$ respectively.

Because of the modularity of the gradients of odd theta functions, both $C_{\ep\,\de}$ and $A_{\ep\,\de}$ are vector-valued modular forms with respect to the group $\Gamma_g(4,8)$ (a more careful analysis of the transformation formula in fact shows that it is modular with respect to $\Gamma_g ^*(2,4)$) for the representation $\det\otimes\Sym^2(\operatorname{std})$ --- that is, with highest weight $(3, 1,\dots, 1)$.

Using the fact that both theta functions with characteristic and theta functions of the second order satisfy the heat equation~\eqref{heat equation} one can express $C_{\ep\de}$ in terms of derivatives of second order theta constants, and vice versa.
\begin{lm}[\cite{grsmodd1}]
We have the following identities of vector-valued modular forms:
\begin{equation}\label{CintermsofA}
C_{\ep\de}= \frac{1}{2}\sum\limits_{\al\in\FF_2^g}(-1)^{\al\cdot\de }A_{ \ep+\al \,\al};
\end{equation}
\begin{equation}\label{AintermsofC}
A_{\ep+\al\,\al} =\frac{1}{2^{g-1}}\sum_{\lbrace\de\in\FF_2^g\,\mid\,[\ep,\de]\odd \rbrace}(-1)^{\al\cdot\de}C_{\ep\de}.
\end{equation}
\end{lm}
Of course we have the same identities relating ${\bf A}_{\ep+\alpha\,\alpha}$ and ${\bf C}_{\ep\de}$.

\section{Characterization of decomposable ppav}
We are now ready to prove our first result, on the characterization of decomposable ppav. Indeed, recall that if $\tau=\left(\begin{smallmatrix}\tau_1&0\\ 0&\tau_2\end{smallmatrix}\right)$, with $\tau_i\in\calH_{g_i}$, for $g_1+g_2=g$, then the theta function with characteristic splits as a product
$$
 \theta_m(\tau,z)=\theta_{m_1}(\tau_1,z_1)\cdot \theta_{m_2}(\tau_2,z_2),
$$
where $z_i\in\CC^{g_i}$, and we have written $m$ as $m_1\, m_2$, with $m_i\in\FF_2^{2g_i}$. Computing the partial derivatives and evaluating at zero we get
$$ v_m(\tau)=\Big(v_{m_1}(\tau_1)\cdot\theta_{m_2}(\tau_2,0),\,
\theta_{m_1}(\tau_1,0)\cdot v_{m_2}(\tau_2)\Big).
$$
Since $m$ is odd, it follows that precisely one of $m_1$ and $m_2$ is odd, and thus only the corresponding $g_i$ entries of the gradient vector are non-zero. Thus if we arrange the gradients for all odd $m$ in a matrix, it will have a block form, with the two non-zero blocks of sizes $g_i\times 2^{g_i-1}(2^{g_i}-1)$, and two ``off-diagonal'' zero blocks. This is simply to say that the set of gradients of all odd theta functions at a point $\tau$ as above lies in the product of coordinate linear spaces $\CC^{g_1}\cup\CC^{g_2}\subset\CC^g$. Since $\grad_z\theta_m(\tau,z)|_{z=0}$ and $\grad_z\theta_0(\tau,z)|_{z=m}$ differ by a constant factor and thus give the same point in $\PP^{g-1}$, this implies that the images of all the smooth two-torsion points of $\Theta_\tau$ under the Gauss map   lie on $g_1 g_2$ reducible quadrics in $\PP^{g-1}$ written explicitly as
$$
  X_iX_j =0,\qquad \forall 1\leq i\leq g_1<j\leq g.
$$
This is equivalent to these Gauss images all lying on a union of two hyperplanes, and a weaker condition is that they all lie on some quadric (not necessarily a reducible one). We now show that this weak condition is enough to characterize the locus of decomposable ppav, proving one of our two main results.
\begin{proof}[{\bf Proof of  theorem \ref{thm:decomposable}}]
The discussion above proves that for a decomposable ppav with a period matrix $\tau=\left(\begin{smallmatrix}\tau_1&0\\ 0&\tau_2\end{smallmatrix}\right)$ the images of all the odd two-torsion points lie on a quadric. In general if a ppav is decomposable, its period matrix does not need to have this block shape, and would rather be conjugate to it under $\Gamma_g$. Since $v_m(\tau)$ are vector-valued modular forms for the representation $\det^{1/2}\otimes \rm{std}$, they transform linearly under the group action, and hence the condition  that the images of the odd two-torsion points under the Gauss map  lie on a quadric  is preserved  under the action of $\Gamma_g$. Thus for any decomposable ppav  the images of all smooth two-torsion points lying on $\Theta_\tau$ are contained in (many) quadrics.

For the other direction of the theorem we manipulate the gradients to reduce to the characterization of the locus of decomposable ppav given by theorem \ref{thm:decchar}. Indeed, suppose all images of the odd two-torsion points $m$ lie on a quadric with homogenous equation $Q(x_1,\ldots,x_g)$: this is to say that
$$
 Q(v_m)= v_m^t B v_m =0
$$
for all odd $m\in X_\tau[2]$ that are smooth points of $\Theta_\tau$ (where we have denoted by $B$ the matrix of coefficients of $Q$). We thus have
$$
  \tr( v_m^t B v_m )=\tr(B v_mv_m^t )=\tr( B C_{m} )=0
$$
for all odd $m$ (if $m\in\op{Sing}X_\tau$, then $v_m=0$, so $C_m=0$, and this still holds).
Since by \eqref{AintermsofC} each $ A_{\alpha\,\beta}$
is a linear combination of the $C_m$'s, it follows that we also have
$$
 \tr(B A_{\alpha\,\beta}  )=0
$$
for all $\alpha,\beta$, and in particular this implies that the matrix
\begin{equation}\label{A}
 {\bf A}:=({\bf A}_{\alpha\,\beta})_{\alpha\neq\beta\in\FF_2^{g}},
\end{equation}
where each ${\bf A}_{\alpha\,\beta}$ is a column-vector in $\CC^{g(g+1)/2}$, is degenerate. The following lemma in linear algebra shows that this implies that the matrix  $M(\tau)$ in theorem \ref{thm:decchar} is degenerate, and thus that $X_\tau$ is decomposable --- completing the proof of the theorem.
\end{proof}
\begin{lm}
The $\frac{g(g+1)}{2}\times 2^{g-1}(2^g-1)$ matrix ${\bf A}(\tau)$ in~\eqref{A}
has  rank less than $\frac{g(g+1)}{2}$ (i.e.~non-maximal) if and only if the matrix $M(\tau)$ has non-maximal rank.
\end{lm}
\begin{proof}
For $1\leq i\leq j\leq g$,  we denote $M_{ij}$ and ${\bf A}_{ij}$, correspondingly, the $ ( i, j)$  rows of the matrices $M(\tau)$ and ${\bf A}(\tau)$, and denote $M_0$ the first row of $M(\tau)$ (the vector of second order theta constants). We then have
$$
  M_0\wedge M_{ij}={\bf A}_{ij}
$$
where by the wedge we mean taking the  row vector whose entries are all two by two minors of the matrix formed by two row vectors $M_0$ and $ M_{ij}$. If the vectors $ {\bf A}_{\alpha\beta}$ are linearly dependent, this means we have some linear relation $0=\sum a_{ij} {\bf A}_{ij}$ among the rows of ${\bf A}(\tau)$, which is equivalent to
$$
  0=\sum_{i, j} a_{ij}(M_0\wedge M_{ij})=M_0\wedge\left(\sum_{i, j} a_{ij}M_{ij}\right)
$$
and thus $M_0$ must be proportional to $\sum a_{ij}M_{ij}$, so that the matrix $M$ does not have maximal rank.
\end{proof}

\begin{rem}
The proof above shows that in fact a quadric in $\PP^{g-1}$ contains the Gauss images of the two-torsion points on the theta divisor if and only if it contains the entire image of the Gauss map.
\end{rem}

\section{Review of constructions of holomorphic differential forms on Siegel modular varieties in \cite{freitagholdiff1,smholdiff}}
For a finite index subgroup $\Gamma\subset\Gamma_g$ we denote, as before, $\calA_g(\Gamma):=\calH_g/\Gamma$, and we are then interested in constructing non-zero degree $k$ differential forms on it, that is elements of $\Omega^k(\calA_g(\Gamma))$. It is known that for $g\ge 2$:
\[\Omega^k(\calA_g(\Gamma))\cong\Omega^k(\HH_g)^\Gamma,\]
where $\Omega^k(\HH_g)^\Gamma$ is the vector space of elements of $\Omega^k(\HH_g)$ invariant under the action of $\Gamma$.
Whenever $k<N=g(g+1)/2$ and $g\ge2$, such holomorphic differential forms always extend. More precisely, if $\HH_g^0/\Gamma$ is the set of regular points of $\HH_g/\Gamma$, and $\tilde{X}$ denotes the desingularization of the Satake compactification of $\HH_g/\Gamma$, which contains $\HH_g^0/\Gamma$ as an embedded open set, then every holomorphic differential form $\omega\in\Omega^k(\HH_g^0/\Gamma)$ of degree $k<N$ extends to $\tilde{X}$ (see \cite{freitagpommerening}).

Holomorphic differential forms can thus also be thought as vector-valued modular forms for a suitable representation. We have the following fundamental result of Weissauer:
\begin{thm}[\cite{weissauervector}]
The space $\Omega^k(\calA_g(\Gamma))$ is zero unless $k=g\alpha-\alpha(\alpha-1)/2$ for some $0\le\alpha\le g$, in which case
\begin{equation}\label{weissauer}
  \Omega^{\alpha g - \frac 1 2 \alpha(\alpha-1)}(\calA_g(\Gamma))=[\Gamma,\rho_\alpha ]
\end{equation}
is the space of vector-valued modular forms for the representation of $\op{GL}(g,\CC)$ with highest weight $(g+1,\ldots,g+1,\alpha,\ldots,\alpha)$, with $\alpha$ appearing $g-\alpha$ times.
\end{thm}
The case $k=N-1$, corresponding to the representation $\rho_{g-1}$ with highest weight $(g+1,\ldots,g+1,g-1)$, turns out to be of great interest, as it is related to the construction of special divisors on the Satake compactification of Siegel modular varieties. Indeed, Weissauer  \cite{weissauerdivisors} proved that the zero locus $D_h$ of a modular form $h$ on the Satake compactification of $\calA_g(\Gamma)$ is a special divisor if and only if there exists a non vanishing $\omega\in\Omega^{N-1}(\HH_g)^{\Gamma}$ such that $\op{Tr}(\omega(\tau)\partial_\tau h(\tau))$ is identically zero on $D_h$.  Moreover, using theta series with pluriharmonic coefficients, Weissauer \cite{weissauerdivisors} proved that for any $g$ the space $\Omega^{N-1}(\calA_g(\Gamma))$ is non-zero for a suitable $\Gamma$. Such forms can be constructed  as follows
$$
 d\check{\tau}_{ij}=\pm\bigwedge_{1\leq h\leq k\leq g, \, (h,k)\neq (i, j)}d\tau_{hk},
$$
where the sign is chosen in such a way that
$d\check{\tau}_{ij}\wedge d\tau_{ij}=\bigwedge_{1\le i<j\le g} d\tau_{ij}$, see \cite{freitagkorper}.
Then we have
\begin{equation}\label{diff form}
\omega=\op{Tr}(A(\tau)d\check{\tau})=\sum_{1\le i,j\le g} A_{ij}(\tau)d\check{\tau}_{ij},
\end{equation}
with
\begin{equation}\label{rho1}
A(\gamma\cdot\tau)=\op{det}(C\tau+D)^{g+1}\,(C\tau+D)^{-t}A(\tau)(C\tau+D)^{-1}.
\end{equation}
In \cite{freitagholdiff1} Freitag provides a method to construct holomorphic differential $(N-1)$-forms in genus $g$, invariant with respect to any subgroup $\Gamma$ of finite index of the symplectic group $\Gamma_g$ starting from two scalar valued modular forms in genus $g$, both of weight $\frac{g-1}{2}$. We briefly recall this construction and slightly improve  his result.
To simplify the notation, we set
\begin{equation}\label{diff operator}
\partial_{ij}=\frac{1}{2}(1+\delta_{ij})\frac\partial{\partial\tau_{ij}};\qquad \partial:=(\partial_{ij}).
\end{equation}
For any $I,J\in P_k(X_g)$ with $0\leq k \leq g$, we denote by
$\partial^I_J$ the submatrix of $\partial$ obtained by taking the rows corresponding to the elements in $I$ and the columns corresponding to the elements in $J$:
\begin{equation*}
\partial^I_J=(\partial_{ij})_{\substack{i\in I\\j\in J}}
\end{equation*}
and consequently by $|\partial^I_J|$ the determinant of such submatrix, namely
$|\partial^I_J|=\det(\partial^I_J).$
For $k=0$, we set both $\partial^I_J$ and $|\partial^I_J|$ to be the identity operator.

Then for any congruence subgroup $\Gamma$, Freitag \cite{freitagholdiff1} defines the linear pairing $\{\,,\,\}$ by
$$
\begin{aligned}
\left\lbrace\ , \ \right\rbrace: \left[\Gamma,(g-1)/2 \right]\times\left[\Gamma,(g-1)/2 \right]&\to\Omega^{N-1}(\calA_g(\Gamma))\\
(f,h)&\mapsto\{f,h\}:=\op{Tr}\left(B(\tau)d\check{\tau}\right),
\end{aligned}
$$
where
\begin{equation*}
B(\tau)_{ij}:=(-1)^{i+j}\sum_{k=0}^{g-1}\frac{(-1)^{k}}{\binom{g-1}{k}}\sum_{
\begin{smallmatrix}
I\in P_{k}^*(X_g\setminus\{i\}) \\
J\in P_{k}^*(X_g\setminus\{j\})
\end{smallmatrix}
}  s(I)s(J)\left|\partial_{J}^{I}\right|f(\tau)\,\left|\partial_{J^c}^{I^c}\right|h(\tau),
\end{equation*}
and $s(I)$ (resp.~$s(J)$) denotes the sign of the permutation of the elements of $X_g\setminus\lbrace i\rbrace$ (resp.~$X_g\setminus\lbrace j\rbrace$) that turns the set $I\cup I^c$ (resp.~$J\cup J^c$) into an  increasing ordered set. One then easily checks that the parity of the pairing is $\{f,h\}=(-1)^{g+1}\{h,f\}$.

In \cite{freitagkorper} Freitag then proved that the holomorphic differential form
\begin{equation}
 F^{(g)}:=\left\lbrace{\sum_m \theta_m^{g-1}(\tau),\sum_m \theta_m^{g-1}(\tau)}\right\rbrace
\end{equation}
does not vanish identically when $g\equiv 1 \pmod 8$, for $g\ge 17$. We extend this result to $g=9$:
\begin{prop}
The vector-valued modular form $F^{(9)}$ does not vanish identically, and thus gives a non-zero differential form in $\Omega^{35}(\calA_9)$.
\end{prop}
\begin{proof}
Since the set of all $d\check\tau_{ij}$ for $1\le i\le j\le g$ is a basis of $\Omega^{N-1}(\calH_g)$, it suffices to prove that at least one $B(\tau)_{ij}$ is not identically zero. By Freitag's computation \cite[eg.~61]{freitagkorper}, the Fourier coefficient of the pairing $\lbrace f,h\rbrace$ with respect to a matrix $T$ is given by
\begin{equation} \label{fF}
a_{\lbrace f,h\rbrace}(T)_{gg}=\sum_{k=1}^g\frac{(-1)^{k}}{\binom{g-1}{k-1}}\sum_{
\begin{smallmatrix}
I,J\in P_{k-1}^*(X_{g-1}) \\
T_1+T_2=T
\end{smallmatrix}
}
s(I)s(J)|T_1|_{J}^{I}|T_2|_{J^c}^{I^c}a_f(T_1)a_h(T_2),
\end{equation}
where $I^c=X_{g-1}\setminus I$ denotes the complement, and $a_f(T_1)$ and $a_h(T_2)$ are the Fourier coefficients of $f$ and $h$ corresponding to the matrices $T_1$ and $T_2$ respectively.

For our case this formula can be greatly simplified. Indeed,  we recall the result of Igusa \cite{igusachristoffel} that $\sum_{m} \theta_m^{8}(\tau)=2^g\Theta_{E_8}^{(g)}$. We then choose
$ T:=
\left(\begin{smallmatrix}
\zeta_{E_8} & 0 \\
0 & 0
\end{smallmatrix}\right),
$
where $\zeta_{E_8}$ is the matrix associated with the quadratic form corresponding to the $E_8$ lattice, given in a suitable basis by
\begin{equation}
\zeta_{E_8}:=
\left(\begin{smallmatrix}
2 & 0 & 0 & 1 & 0 & 0 & 0 & 0 \\
 0 & 2 & 1 & 0 & 0 & 0 & 0 & 0 \\
 0 & 1 & 2 & 1 & 0 & 0 & 0 & 0 \\
 1 & 0 & 1 & 2 & 1 & 0 & 0 & 0 \\
 0 & 0 & 0 & 1 & 2 & 1 & 0 & 0 \\
 0 & 0 & 0 & 0 & 1 & 2 & 1 & 0 \\
 0 & 0 & 0 & 0 & 0 & 1 & 2 & 1 \\
 0 & 0 & 0 & 0 & 0 & 0 & 1 & 2
\end{smallmatrix}\right).
\end{equation}
By K\"ocher principle, the Fourier coefficients $a_f(S)$ or $a_h(S)$ with respect to a non-semidefinite positive matrix $S$  are zero, and thus only the terms with even semidefinite positive $T_1$ and $T_2$ produce non-zero summands in \eqref{fF}. Whenever the chosen $T$ is written as $T=T_1+T_2$ with $T_1,T_2$ positive semidefinite matrices,  one of $T_i$ must be zero. Finally, recall that for $g=9$ we have
$$
\Theta_{E_8}(\tau)=\!\!\!\sum_{x_1,\ldots,x_9\in\Lambda_{E_8}}\!\!\!e^{\pi i{\rm Tr}(x\cdot
  x)}=\!\!\sum_{p\in\ZZ^{g=9,8}}e^{\pi i{\tr}(p \zeta_{E_8} p^t\tau)}
  =\sum_{M} N_{M}\prod_{i\leq j}e^{\pi i m_{ij}\tau_{ij}},
$$
where, for $M=(m_{ij})$ a symmetric $g\times g$ integer matrix, $N_M\in\NN$ is the number of integral matrix solutions of the Diophantine system $p\zeta_{E_8}p^t=M$.
Setting $M=T$ and writing
$p=\left(\begin{smallmatrix}p_1 \\ p_2 \end{smallmatrix}\right)$, where $p_1$ and $p_2$ are respectively $8\times 8$ and $1\times 8$ integer matrices, it follows that for all solutions $p_2=0$, while $p_1$ satisfies  $p_1 \zeta_{E_8} p_1^t=\zeta_{E_8}$.

The number of solutions of the previous equations equals the order of the group $U(\zeta_{E_8})$ of automorphisms of the $E_8$ lattice, i.e. $a(\zeta_{E_8})=\#(U(\zeta_{E_8}))=4!6!8!$, see \cite[page 121]{conwaysloanebook}. Thus we finally have $N_T=a_{F^{(9)}}(T)_{99}=4!6!8!$, hence there is a non-empty set of summands in \eqref{fF}, all of them positive, so it follows that $A(T)_{99}$ is non-zero.
\end{proof}
\begin{rem}
The argument above generalizes to give an alternative proof of Freitag's result for any $g=8k+1$, for $k\ge 1$, using the modular form $\Theta_{E_8}(\tau)^{k}$.
\end{rem}

\medskip
We now recall another construction of holomorphic differentials forms, due to the fifth author \cite{smholdiff}.
For $M=(m_1,\dots,m_{g-1})$ a set of distinct odd characteristics define
$$
  F(m_1,\dots,m_{g-1})(\tau):=v_{m_1}(\tau)\wedge\ldots\wedge v_{m_{g-1}}(\tau).
$$
One can then use these wedge products of gradients of theta functions to construct further vector-valued modular forms. We set
\begin{equation}\label{Wform}
 W(M)(\tau):=\pi^{-2g+2} F(m_1,\dots,m_{g-1})(\tau)^t\,F(m_1,\dots,m_{g-1})(\tau)
\end{equation}
and then have
\begin{prop}[\cite{smholdiff}] \label{lemma4}
For $g$ odd, for any matrix of distinct odd characteristics  $M=(m_1,\dots, m_{g-1} )\in M_{2g\times (g-1)}(\FF_2)$
\begin{equation*}
\omega(M)(\tau):=\tr\left(W(m_1,\dots,m_{g-1})(\tau)  d\check{\tau}\right)
\end{equation*}
is a non-zero holomorphic differential form     in  $\Omega^{N-1}(\calA_g(2,4))$.  If $g$ is even, it is a non-zero holomorphic differential form  in  $\Omega^{N-1}(\calA_g^*(2,4))$
\end{prop}

\begin{rem}
Symmetrizing the $\omega(M)$ constructed above using the action of the entire modular group,   differential forms for the entire modular group were obtained in \cite{smholdiff},  thus showing that $\Omega^{N-1}(\calA_g)$ is non-zero for any $g\equiv 1\pmod 4$, $g\ne 1,5,13$.
\end{rem}

\section{A new construction of differential forms}

Our first main theorem, Theorem \ref{thm:holdiff}, gives an easy new method to construct non-zero holomorphic differential forms on Siegel modular varieties, using the modular forms $A_{\ep\de}$. We prove that it works, and then relate this new construction to the two constructions discussed above.
\begin{proof}[{\bf Proof of theorem \ref{thm:holdiff}}]
Recall that for fixed $\ep,\de$ the matrix $A_{\ep\de}$ can be written as
$$A_{\ep\,\delta}(\tau):=4\pi i\, \Theta[\de]^2 \partial\left(\frac{\Theta[\ep]} {\Theta[\delta]}\right),$$
and thus its entries are vector-valued  modular forms for the representation of highest weight $(3,1,\dots,1)$.

We denote by $A^{ad}_{\ep\,\de}$ the adjoint matrix --- the transpose of the matrix of cofactors of $A$. This matrix is then clearly a vector-valued modular form  $A^{ad}_{\ep\,\de}\in[\Gamma ,(g+1,\dots,g+1,g-1)]$  with $\Gamma= \Gamma_g (2,4)$ for $g$ odd, and $\Gamma=\Gamma_g^*(2,4)$ for $g$ even, and thus $\op{Tr}( A^{ad}_{\ep\,\de}\,d\check\tau)$ defines a differential form of degree $N-1$ as claimed. It remains to prove that this differential form is not identically zero. Recalling that the product of a matrix and the matrix of its cofactors is the determinant times the identity matrix, if we prove that $\det A_{\ep\,\de}$ is not identically zero, it would follows that $A^{ad}_{\ep\,\de}$ is not identically zero and thus that $\tr ( A^{ad}_{\ep\,\de}\,d\check\tau)$ is not identically zero. The proof is thus completed by the following proposition.
\end{proof}
\begin{prop}
The determinant $\det A_{\ep\,\delta}$ is a not identically zero scalar modular form of weight $g+2$.
\end{prop}
\begin{proof} 
Since  $\Theta[\ep]$ and $\Theta[\de] $ are different forms, there exist
$\tau$ such that $\Theta[\ep](\tau)=0\ne\Theta[\de](\tau)$ . We then denote $Z:=2\tau$, and work on the abelian variety $X_Z$, where $Z\ep/2\in\Theta_Z$ and $Z\de/2\not\in\Theta_Z$ are thus two-torsion points.
Since the characteristics are even, the point  $Z\ep/2$ is then an even two-torsion point lying on $\Theta_Z$, and thus is a singular point of $\Theta_Z$. From \cite{grsmconjectures} it follows that generically  the singularity of $\Theta_Z$ at $Z\ep/2$ is an ordinary double point.  This is equivalent, via the  heat equations, to  the matrix $\partial \theta_m(Z, 0)$, with  $m=[\ep, 0]$, having rank $g$.  Moreover,  we choose $Z$ such that $\theta_n(Z)\neq 0$,   with   $n=[\de, 0]$ and thus see that $\det A$ is not identically zero.
\end{proof}

We will now compare the different construction of modular forms. In Freitag's construction, let us consider Freitag's pairing when $f$ and $h$ are suitable powers of second order theta constants.
For any $\ep\ne\de\in\FF_2^g$ let
\begin{equation}\label{omegadefn}
\omega_{\ep\,\de}:=\{\T[\ep]^{g-1},\T[\de]^{g-1}\},
\end{equation}
and then a simple computation on the characters shows that for $g$ odd $\omega_{\ep\,\de}\in\Omega^{N-1}(\calA_g(2,4))$, while for $g$ even we only get $\omega_{\ep\,\de}\in\Omega^{N-1}(\calA_g^*(2,4))$ for the quotient corresponding to the index two subgroup $\Gamma_g^*(2,4)\subset\Gamma_g(2,4)$.

To relate this to the current construction, we first prove the following
\begin{prop}
For any $\ep\neq\de$  we have
\[A_{\ep\,\de}^{ad}( \tau)=\left(\frac{\pi^2}{2^{g-2}}\right)^{g-1}\!\!\!\!\!\!\!\!\sum_{\substack{\alpha_{1},\dots,\alpha_{g-1}\in\FF_2^g\\ \op{s.t.}\,[\ep+\de,\,\alpha_{j}]\op{odd}}}\!\!\!\!\!(-1)^{\de\cdot(\alpha_{1}+\cdots+\alpha_{g-1})}W([\ep+\de,\,\alpha_{1}],\dots ,[\ep+\de,\,\alpha_{g-1}]),\]
where $W$ is defined in \eqref{Wform}.
\end{prop}
\begin{proof}
We will need some basic facts from linear algebra. Let $A$ and $B$ be a $m\times n$ and a $n\times m$ matrices respectively, then
\begin{equation} \label{fact1}
AB=\sum_{i=1}^nA_iB^i,
\end{equation}
where $A_i$ is the $i$-th column of $A$ and $B^i$ is the $i$-th row of $B$.
For $I,J\in P_k^*(X_m)$, then the following holds:
\begin{equation}\label{fact2}
(AB)^I_J=A^IB_J,
\end{equation}
where $A^I$ is the submatrix obtained from $A$ by taking rows corresponding to the elements of $I$ and $B_J$ is the submatrix obtained from $B$  by taking columns corresponding to the elements of $J$. The last identity we need is the following generalization of the Binet formula:
\begin{equation}\label{fact3}
\det(AB)=\sum_{S\in P_m^*(X_n)}\det(A_S)\,\det(B^S).
\end{equation}
Notice that if $m>n$, $P_m^*(X_n)$ is empty and the right side of the previous identity is zero, as should be the case, since the rank of $AB$ is bounded by the ranks of $A$ and $B$.
Defining the $g\times 2^g$ matrix
\[V_{\ep+\de}=\left(v_{\ch{\ep+\de}{\al}}\right)_{\al\in\FF_2^g},\]
whose columns are the gradients $v_{\ch{\ep+\de}{\al}}$ indexed by $\al\in\FF_2^g$, and the $2^g\times g$ matrix
\[V_{\ep+\de}^-=\left((-1)^{\de\cdot\al}\,v_{\ch{\ep+\de}{\al}}^t\right)_{\al\in\FF_2^g},\]
relations \eqref{AintermsofC} and \eqref{fact1} imply:
\[A_{\ep,\de}=\frac{1}{2^{g-2}}\,V_{\ep+\de}\,V_{\ep+\de}^-.\]
Hence, by a straightforward computation from~\eqref{fact2} and~\eqref{fact3} the 
proposition follows.
\end{proof}
We now compare our construction to that of Freitag, thus also linking the two previously known methods.
\begin{thm}\label{prop adjoint}
For $\ep\ne\de$ denote by $B_{\ep\,\de}$ the vector-valued modular form such that
$\{\Theta[\ep]^{g-1},\Theta[\de]^{g-1}\}=\tr(B_{\ep\,\de}(\tau)d\check{\tau})$. Then we have
\begin{equation}\label{adjoint}
  A_{\ep\,\de}^{ad}=\frac{(4\pi i)^{g-1}}{(g-1)!}B_{\ep\,\de}.
\end{equation}
\end{thm}
We note that of course the above is an identity of vector-valued modular forms, which also implies that the holomorphic differential forms constructed from them are equal in $\Omega^{N-1}(\calA_g(2,4))$ and $\Omega^{N-1}(\calA_g^*(2,4))$ for $g$ odd and even respectively.

The proof of Theorem~\ref{prop adjoint} relies on the following
\begin{lm}\label{lemma_par}
Let $I=\{i_1,\dots,i_k\}$, $J=\{j_1,\dots,j_k\}$ be elements of $P^*_k(X_{g})$ with $k\leq n$. As a consequence of the heat equations, for every $\ep\in\FF_2^g$ the second order theta constant $\T[\ep]$ satisfies the relation
\[|\partial^I_J|\,\Theta[\ep]^n=n(n-1)\cdots(n-k+1)\Theta[\ep]^{n-k}\,|(\partial\Theta[\ep])^I_J|.\]
\end{lm}
\begin{rem}
We emphasize that the left-hand-side of the lemma means the determinant of the matrix of partial derivatives, considered as a degree $k$ differential operator, applied to the power of the theta constant, while the right-hand-side is a different power of the theta constant multiplied by the determinant of the matrix of partial derivatives of the theta constants. When differentiating on the left, one would a priori expect terms involving higher order derivatives of the theta constant to appear, and the content of the lemma is that such cancel out.
\end{rem}
\begin{proof}
The proof will be done by induction on $k$. Clearly, for $k=1$
\[\frac{(1+\delta_{i_1 j_1})}{2}\partialtau {i_1}{j_1}\Theta[\ep]^{n}=n\Theta[\ep]^{n-1}\frac{(1+\delta_{i_1 j_1})}{2}\partialtau {i_1}{j_1}\Theta[\ep].\]
The first interesting case is $k=2$, where $I=\{i_1,i_2\}$ and $J=\{j_1,j_2\}$. In this case we have
\[|\partial^I_J|\,\Theta[\ep]^n=n(n-1)\Theta[\ep]^{n-2}\,|(\partial\Theta[\ep])^I_J|\,+n\Theta[\ep]^{n-1}
(|\partial^I_J|\,\Theta[\ep]).\]
From the heat equation it easily follows that for every $\ep\in\FF_2^g$
\[(1+\delta_{i_1 j_1})(1+\delta_{i_2 j_2})\partialtau {i_1}{j_1}\partialtau {i_2}{j_2}\Theta[\ep]=
(1+\delta_{i_2 j_1})(1+\delta_{i_1 j_2})\partialtau {i_2}{j_1}\partialtau {i_1}{j_2}\Theta[\ep], \]
hence
\begin{equation}\label{2by2}
|\partial^I_J|\,\Theta[\ep]=
\begin{vmatrix}
\frac{(1+\delta_{i_1 j_1})}{2}\partialtau {i_1}{j_1} & \frac{(1+\delta_{i_1 j_2})}{2}\partialtau {i_1}{j_2}   \\
\frac{(1+\delta_{i_2 j_1})}{2}\partialtau {i_2}{j_1}  & \frac{(1+\delta_{i_2 j_2})}{2}\partialtau {i_2}{j_2}
\end{vmatrix}\Theta[\ep]=0.
\end{equation}

Computing $|\partial^I_J|$ by the Laplace expansion along the first column for $k>2$, we have
\begin{align*}
&\ |\partial^I_J|\,\Theta[\ep]^n=
\Big(\sum_{h=1}^k(-1)^{h+1}\partial_{i_hj_1}\left|\partial^{I\setminus\{i_h\}}_{J\setminus\{j_1\}}\right|\,\Big)\Theta[\ep]^n=\\
&=\sum_{h=1}^k(-1)^{h+1}\partial_{i_hj_1}
\left[n(n-1)\cdots(n-k+2)
\Theta[\ep]^{n-k+1}
\left|(\partial\Theta[\ep])^{I\setminus\{i_h\}}_{J\setminus\{j_1\}}\right|\right]=\\
&=n(n-1)\cdots(n-k+1)\Theta[\ep]^{n-k}\,|(\partial\Theta[\ep])^I_J|+\\
&\quad+n(n-1)\cdots(n-k+2)
\Theta[\ep]^{n-k+1}\sum_{h=1}^k(-1)^{h+1}\partial_{i_hj_1}\,
\left|(\partial\Theta[\ep])^{I\setminus\{i_h\}}_{J\setminus \{j_1\}}\right|.
\end{align*}
The extra terms cancel out because of the heat equation, so the lemma is proved.
\end{proof}

We are now ready to prove the about theorem.
\begin{proof}[{\bf Proof of theorem \ref{prop adjoint}}]
By \cite[lemma~4]{weissauervector}, to prove the identity of such vector-valued modular forms, it is enough to prove that, for example, the $gg$ entries of the corresponding matrices agree.

We first recall  that the determinant of a matrix can be expanded in its block submatrices as follows: for an $n\times n$ matrix $M$, and for any fixed $J\in P_k^*(X_n)$, we have
$$
 \det(M)=\sum_{I\in P_k^*(X_n)}(-1)^{I+J}\cdot|M^I_J|\cdot |M^{I^c}_{J^c}|
$$
where on the right we take the determinants of the corresponding submatrices, and $(-1)^I$ means $(-1)^{i_1+\ldots+i_k}$ where $I=\lbrace i_1,\ldots,i_k\rbrace$.
Applying this to $gg$-th entry of the cofactor matrix, we get
$$
(A_{\ep\,\de}^{ad})_{gg}=(4\pi i)^{g-1}\sum_{k=0}^{g-1}(-1)^{k}\Theta[\ep]^{g-k-1}\Theta[\de]^{k}\cdot\hskip4cm$$
$$\hskip4cm\cdot\sum_{I,J\in P^*_k(X_{g-1})}(-1)^{I+J}|(\partial\Theta[\ep])^I_J|\cdot|(\partial\Theta[\de])^{I^c}_{J^c}|.
$$
By Lemma \ref{lemma_par} it follows that
$$
(B_{\ep\,\de})_{gg}=(g-1)!\sum_{k=0}^{g-1}(-1)^{k}\Theta[\ep]^{g-k-1}\Theta[\de]^{k}\cdot\hskip4cm$$
$$\hskip4cm\cdot\sum_{I,J\in P^*_k(X_{g-1})} s(I)s(J)|(\partial\Theta[\ep])^{I}_{J}|\cdot |(\partial\Theta[\de])^{I^c}_{J^c}|.
$$

To complete the proof it is enough to check that
$s(I)\,s(J)=(-1)^{I+J}.$
This can be easily verified by induction on $k$ noting that for $I=\{i\}$ it holds that $s(I)=(-1)^{i-1}$ since it is the sign of the permutation that turns the set $\{i,1,\dots,i-1,i+1,\dots,g-1\}$ into the set $\{1,\dots,g-1\}$.
\end{proof}

\begin{rem}
In all of the constructions above instead of starting from $A_{\ep\de}$, 
one can perform the same construction starting from theta constants of arbitrary level or from two theta constants with characteristic. As a result one gets vector-valued modular forms for suitable subgroups which can be used to construct holomorphic differential forms on suitable Siegel modular varieties.
\end{rem}


\begin{thebibliography}{AMRT10}

\bibitem[AMRT10]{amrtbook}
A.~Ash, D.~Mumford, M.~Rapoport, and Y.~Tai.
\newblock {\em Smooth compactifications of locally symmetric varieties}.
\newblock Cambridge Mathematical Library. Cambridge University Press,
  Cambridge, second edition, 2010.
\newblock With the collaboration of Peter Scholze.

\bibitem[CS99]{conwaysloanebook}
J.~H. Conway and N.~J.~A. Sloane.
\newblock {\em Sphere packings, lattices and groups}, volume 290 of {\em
  Grundlehren der Mathematischen Wissenschaften}.
\newblock Springer-Verlag, New York, third edition, 1999.

\bibitem[CS03a]{case2}
L.~Caporaso and E.~Sernesi.
\newblock Characterizing curves by their odd theta-characteristics.
\newblock {\em J. Reine Angew. Math.}, 562:101--135, 2003.

\bibitem[CS03b]{case1}
L.~Caporaso and E.~Sernesi.
\newblock Recovering plane curves from their bitangents.
\newblock {\em J. Algebraic Geom.}, 12(2):225--244, 2003.

\bibitem[FP82]{freitagpommerening}
E.~Freitag and K.~Pommerening.
\newblock Regul\"are {D}ifferentialformen des {K}\"orpers der {S}iegelschen
  {M}odulfunktionen.
\newblock {\em J. Reine Angew. Math.}, 331:207--220, 1982.

\bibitem[Fre75a]{freitagholdiff1}
E.~Freitag.
\newblock Holomorphe {D}ifferentialformen zu {K}ongruenzgruppen der
  {S}iegelschen {M}odulgruppe.
\newblock {\em Invent. Math.}, 30(2):181--196, 1975.

\bibitem[Fre75b]{freitagholdiff2}
E.~Freitag.
\newblock Holomorphe {D}ifferentialformen zu {K}ongruenzgruppen der
  {S}iegelschen {M}odulgruppe zweiten {G}rades.
\newblock {\em Math. Ann.}, 216(2):155--164, 1975.

\bibitem[Fre78]{freitagkorper}
E.~Freitag.
\newblock Der {K}\"orper der {S}iegelschen {M}odulfunktionen.
\newblock {\em Abh. Math. Sem. Univ. Hamburg}, 47:25--41, 1978.
\newblock Special issue dedicated to the seventieth birthday of Erich
  K{\"a}hler.

\bibitem[Fre91]{freitagbooksingular}
E.~Freitag.
\newblock {\em Singular modular forms and theta relations}, volume 1487 of {\em
  Lecture Notes in Mathematics}.
\newblock Springer-Verlag, Berlin, 1991.

\bibitem[GH11]{grhu2}
S.~Grushevsky and K.~Hulek.
\newblock Principally polarized semiabelic varieties of torus rank up to 3, and
  the {A}ndreotti-{M}ayer loci.
\newblock {\em Pure and Applied Mathematics Quarterly, special issue in memory
  of Eckart Viehweg}, 7:1309--1360, 2011.

\bibitem[GH12]{grhu1}
S.~Grushevsky and K.~Hulek.
\newblock The class of the locus of intermediate {J}acobians of cubic
  threefolds.
\newblock {\em Invent. Math.}, 190(1):119--168, 2012.

\bibitem[GSM04]{grsmodd1}
S.~Grushevsky and R.~Salvati~Manni.
\newblock Gradients of odd theta functions.
\newblock {\em J. Reine Angew. Math.}, 573:45--59, 2004.

\bibitem[GSM05]{grsmjacformula}
S.~Grushevsky and R.~Salvati~Manni.
\newblock Two generalizations of {J}acobi's derivative formula.
\newblock {\em Math. Res. Lett.}, 12(5-6):921--932, 2005.

\bibitem[GSM06]{grsmodd2}
S.~Grushevsky and R.~Salvati~Manni.
\newblock Theta functions of arbitrary order and their derivatives.
\newblock {\em J. Reine Angew. Math.}, 590:31--43, 2006.

\bibitem[GSM09]{grsmconjectures}
S.~Grushevsky and R.~Salvati~Manni.
\newblock The loci of abelian varieties with points of high multiplicity on the
  theta divisor.
\newblock {\em Geom. Dedicata}, 139:233--247, 2009.

\bibitem[Igu72]{igusabook}
J.-I. Igusa.
\newblock {\em Theta functions}, volume 194 of {\em Grundlehren der
  Mathematischen Wissenschaften}.
\newblock Springer-Verlag, New York, 1972.

\bibitem[Igu81]{igusachristoffel}
J.-I. Igusa.
\newblock Schottky's invariant and quadratic forms.
\newblock In {\em E. {B}. {C}hristoffel ({A}achen/{M}onschau, 1979)}, pages
  352--362. Birkh\"auser, Basel, 1981.

\bibitem[Sas83]{sasaki}
R.~Sasaki.
\newblock Modular forms vanishing at the reducible points of the {S}iegel
  upper-half space.
\newblock {\em J. Reine Angew. Math.}, 345:111--121, 1983.

\bibitem[SM87]{smholdiff}
R.~Salvati~Manni.
\newblock Holomorphic differential forms of degree {$N-1$} invariant under
  {$\Gamma\sb g$}.
\newblock {\em J. Reine Angew. Math.}, 382:74--84, 1987.

\bibitem[SM94]{smlevel2}
R.~Salvati~Manni.
\newblock Modular varieties with level {$2$} theta structure.
\newblock {\em Amer. J. Math.}, 116(6):1489--1511, 1994.

\bibitem[Tai82]{Tai}
Y.-S. Tai.
\newblock On the {K}odaira dimension of the moduli space of abelian varieties.
\newblock {\em Invent. Math.}, 68(3):425--439, 1982.

\bibitem[Wei83]{weissauervector}
R.~Weissauer.
\newblock Vektorwertige {S}iegelsche {M}odulformen kleinen {G}ewichtes.
\newblock {\em J. Reine Angew. Math.}, 343:184--202, 1983.

\bibitem[Wei87]{weissauerdivisors}
R.~Weissauer.
\newblock Divisors of the {S}iegel modular variety.
\newblock In {\em Number theory ({N}ew {Y}ork, 1984--1985)}, volume 1240 of
  {\em Lecture Notes in Math.}, pages 304--324. Springer, Berlin, 1987.

\end{thebibliography}

\end{document}